\documentclass[reqno, 12pt, reqno]{amsart}

\usepackage{amsfonts, amsthm, amsmath, amssymb}
\usepackage{hyperref}
\usepackage[all]{xy} 
\usepackage[margin=1.5in]{geometry}
\usepackage{multirow}
\usepackage{helvet}
\usepackage{stackrel}
\RequirePackage{mathrsfs} \let\mathcal\mathscr
\usepackage{verbatim} 

\makeatletter
\def\@tocline#1#2#3#4#5#6#7{\relax
 \ifnum #1>\c@tocdepth 
 \else
 \par \addpenalty\@secpenalty\addvspace{#2}%
 \begingroup \hyphenpenalty\@M
 \@ifempty{#4}{%
 \@tempdima\csname r@tocindent\number#1\endcsname\relax
 }{%
 \@tempdima#4\relax
 }%
 \parindent\z@ \leftskip#3\relax \advance\leftskip\@tempdima\relax
 \rightskip\@pnumwidth plus4em \parfillskip-\@pnumwidth
 #5\leavevmode\hskip-\@tempdima
 \ifcase #1
 \or\or \hskip 1em \or \hskip 2em \else \hskip 3em \fi%
 #6\nobreak\relax
 \dotfill\hbox to\@pnumwidth{\@tocpagenum{#7}}\par
 \nobreak
 \endgroup
 \fi}
\makeatother

\usepackage[usenames,dvipsnames]{color}

\usepackage{fancyref} 

\usepackage[utf8]{inputenc}
\usepackage{bm}

\numberwithin{equation}{section}
\newtheorem{theorem}{Theorem}[section]
\newtheorem*{theorem*}{Theorem}
\newtheorem{lemma}[theorem]{Lemma}
\newtheorem*{lemma*}{Lemma}

\newtheorem{corollary}[theorem]{Corollary}
\newtheorem*{corollary*}{Corollary}

\newtheorem*{conjecture*}{Conjecture}
\theoremstyle{definition}

\newtheorem*{remark*}{Remark}

\newtheorem*{definition*}{Definition}
\newtheorem{example}[theorem]{Example}
\newtheorem*{example*}{Example}
\usepackage[english]{babel}
\usepackage{enumerate} 
\usepackage{relsize} 
\allowdisplaybreaks 

\DeclareMathOperator{\Gal}{Gal} 
\DeclareMathOperator{\Hilb}{Hilb} 
\DeclareMathOperator{\Sym}{Sym} 
\DeclareMathOperator{\Pic}{Pic} 
\DeclareMathOperator{\rk}{rk} 
\DeclareMathAlphabet{\xcal}{OMS}{cmsy}{m}{n} 
\DeclareMathOperator{\ch}{char} 

\newcommand{\Fq}{\mathbb{F}_{q}}

\newcommand*\diff{\mathop{}\!\mathrm{d}}

\newcommand{\mbC}{{\mathbb C}}

\newcommand{\mbN}{{\mathbb N}}

\newcommand{\mbP}{{\mathbb P}}
\newcommand{\mbQ}{{\mathbb Q}}
\newcommand{\mbR}{{\mathbb R}}

\newcommand{\mbZ}{{\mathbb Z}}

\newcommand{\mfS}{{\mathfrak S}}

\renewcommand{\leq}{\leqslant}
\renewcommand{\geq}{\geqslant}

\begin{document}

\title[Counting points on $\Hilb^m\mbP^2$ over function fields]{Counting points on $\Hilb^m\mbP^2$ over function fields}

\author{Adelina\ M\^{a}nz\u{a}\c{t}eanu}
\address{ School of Mathematics\\University of Bristol\\Bristol\\BS8 1TW\\UK \& IST Austria\\Am Campus 1\\3400 Klosterneuburg\\Austria}

\email{am15112@bristol.ac.uk}

\begin{abstract}
Let $K$ be a global field of positive characteristic. We give an asymptotic formula for the number of $K$-points of bounded height on the Hilbert scheme $\Hilb^2\mbP^2$ and show that by eliminating an exceptional thin set, the constant in front of the main term agrees with the prediction of Peyre in the function field setting. Moreover, we extend the analogy between the integers and $0$-cycles on a variety $V$ over a finite field to $0$-cycles on a variety $V$ over $K$ and establish a version of the prime number theorem in the case when $V=\mbP^2$.
\end{abstract}

\date{\today}


\maketitle

\tableofcontents

\thispagestyle{empty}
\section{Introduction}\label{Introduction}

Let $X$ be a smooth projective variety defined over a global field $K$ such that the set $X(K)$ of rational points is Zariski dense and let $\xcal{L}$ be an ample line bundle on $X$. One can define a height function $H_{\xcal{L}}$ on $X(K)$ taking values in $\mbR_{>0}$ in order to study the  asymptotic behaviour of the number of $K$-points of bounded height on $X$. 

Manin and his collaborators (see \cite{BM}, \cite{FMT}) predict that when $X$ is a Fano variety and $K$ is a number field, there exists an open subset $U$ of $X$ such that
\begin{equation*}
\#\left\{ x \in U(K) : H_{\omega^{-1}_X}(x) \leq B \right\} \sim C_{H_{\omega^{-1}_X}} B \left(\log B\right)^{t-1},
\end{equation*}
as $B \to \infty$, where $H_{\omega^{-1}_X}$ is the height associated to the anticanonical line bundle $\omega_X^{-1}$ and $t$ is the rank of the Picard group of $X$. Moreover, Peyre \cite{Peyre95} gave a conjectural interpretation of the constant $C_{H_{\omega_X^{-1}}}$. The first counterexample was provided by Batyrev and Tschinkel \cite{BT1}. This led to a refined version of Manin's conjecture in which one is allowed to remove a finite number of thin sets as defined by Serre \cite[\S 3.1]{Serre08}. Results that support the need for the thin set version of this conjecture have been proven by Le Rudulier \cite[Theorem 4.2]{Rudulier} and Browning--Heath-Brown \cite[Theorem 1.1]{BHB}.

In the case when $K$ is a global field of positive characteristic, Peyre \cite[Theorem 3.5.1]{Peyre12} proposed an analogue of \cite{Peyre95}. It is natural to extend the thin set version of Manin's conjecture to this case. The goal of this paper is to provide an example that supports this conjecture in the case of function fields of positive characteristic. Consider $\Hilb^2\mbP^2$ over a function field $K$ of degree $e$ over $\Fq(t)$ such that $\ch(K)>2$. 

\begin{theorem}\label{MainResult0} 
There exists a non-empty thin set $Z_0\subset \Hilb^2\mbP^2(K)$ such that 
\begin{align*}
\# \left\{z \in  \Hilb^2 \mbP^2(K) \setminus Z_0  : H_{\omega_{X}^{-1}} (z) = q^M \right\} = c  q^{M} M + O\left(\sqrt{M}q^{M}\right),
\end{align*}
as $M\to \infty$, where the leading constant agrees with the prediction of Peyre.
\end{theorem}

We shall present the geometry of the Hilbert scheme in Section \ref{S:Geometry of the Hilbert scheme}, together with a more refined version of Theorem \ref{MainResult0}. This can be seen as a function field analogue of \cite[Theorem 4.2]{Rudulier}. However, the result of Le Rudulier is only valid over $\mbQ$, whereas Theorem \ref{MainResult0} holds over any finite extension of $\Fq(t)$. Le Rudulier's strategy is to use the geometry of $\Hilb^2\mbP^2$ to reduce the problem to counting points in $\mbP^2$ that are quadratic over $\mbQ$. Then, one appeals to work of Schmidt \cite[Theorem 3]{Schmidt2}. Suppose $K$ is a number field of degree $e$ over $\mbQ$, $d$ is a positive integer, and define the quantity
\begin{align*}
N_{K}(n+1,d,B) = \# \left \{ P \in \mbP^{n} (\overline{\mbQ}) : H'(x) \leq B, [K(P):K] \leq d \right \},
\end{align*}
where $H'(x) = \prod_{v \in \Omega_K}  \max_{0\leq i \leq n} |x_i|_v$. Schanuel \cite{Schanuel} proved that 
\begin{align}\label{eq:Schanuel}
N_{K}(n+1,1,B) = S_K(n+1,1)B^{(n+1)e} + O_{K,n}\left(B^{(n+1)e-1}\right),
\end{align}
with an additional $\log B$ term in the error when $K = \mbQ$ and $n = 1$. The leading constant is called the \emph{Schanuel constant} 
\begin{align}\label{eq:SchanuelsConstant}
S_K(n+1,1) = \left(\frac{2^{r}(2\pi)^s}{|\Delta|^{1/2}}\right)^{n+1} \frac{(n+1)^{r+s-1} hR}{w\zeta_K(n+1)},
\end{align}
where $r$ is the number of real embeddings of $K$, $s$ the number of pairs of distinct complex conjugate embeddings of $K$, $\Delta$ the discriminant of $K$, $h$ is the class number of $K$, $R$ the regulator of $K$, $w$ the number of roots of unity in $K$, and $\zeta_K$ the Dedekind zeta-function of $K$. Schmidt \cite{Schmidt2} generalised this result for quadratic number fields. More precisely, the result we are interested in states that
\begin{align}\label{eq:Schmidt}
N_{\mbQ}(3,2,B) = \frac{24+2\pi^2}{\zeta(3)^2}B^6 \log B + O\left(B^6 \sqrt{\log B}\right),
\end{align}
where the leading constant is a sum of Schanuel constants over extensions of degree $d=2$ over $K=\mbQ$.

The method of proof for Theorem \ref{MainResult0} is similar, and begins with a function field analogue of \eqref{eq:Schmidt}. Given a function field $K$, let $e = [K: \Fq(t)]$, and define the quantity
\begin{align*}
N_K(n+1,d,M) = \# \left\{ P \in \mbP^{n}(\overline{\Fq(t)}) : H_n(P) = q^\frac{M}{ed}, [K(P):K]=d \right\},
\end{align*}
where $H_n$ denotes the usual absolute height on projective space $\mbP^{n}$ with respect to the ground field $\Fq(t)$. More precisely, given a point $x \in \mbP^{n}(\overline{\Fq(t)})$ of degree $d$ over $K$ with homogeneous coordinates $ \left [ x_0 : \ldots : x_n \right ] $ we have
\begin{align}\label{eq:FctFieldHeight}
H_n(x) =\left( \prod_{v \in \Omega_{K(x)}}  \max_{0\leq i \leq n} |x_i|_v\right)^{\frac{1}{de}},
\end{align}
where $K(x)$ is the field obtained by adjoining all quotients $x_i/x_j$ and $\Omega_{K(x)}$ is the set of places of $K(x)$. For details on why this is the right choice of height for such a point see \cite{TW} and \cite[\S 15.1]{A}. The analogue of \eqref{eq:Schanuel} was initially stated by Serre \cite[Section 2.5]{Serre89} who gave a formula for the constant
\begin{align}\label{eq:SchanuelsConstantFctFields}
S_K(n+1,1) = \left(\frac{1}{q^{g_K-1}}\right)^{n+1} \frac{J_K}{(q-1)\zeta_K(n+1)},
\end{align}
where $g_K$ is the genus of $K$, $J_K$ is the number of divisor classes of degree 0 which is the cardinality of the Jacobian of $K$ and $\zeta_K$ is the zeta function of $K$. Later Wan \cite{Wan} and DiPippo \cite{DiPippo} independently gave a proof of this result. Further work has been done by Thunder and Widmer \cite{TW} who obtain results for $d\geq 1$ and a more precise error term in the case $d=1$. In particular, they prove that for $M \geq 2g_K -1$ and $1/4 \geq \varepsilon >0$, we have
 \begin{align}\label{eq:TW}
 N_K(3,1,M) = S_K(3,1)q^{3M} + O\left(q^{(M+g_K)(1+\varepsilon)}\right),
\end{align}
and if $M <2g_K-1$, then for all $\varepsilon >0$, we have
 \begin{align}\label{eq:TW2}
N_K(3,1,M) \ll q^{M(2+\varepsilon)},
\end{align}
where the implicit constants depend only on $n, e, q$ and $\varepsilon$. In the case when $d=2$, Kettlestrings and Thunder \cite{KT} improved the result of Thunder and Widmer \cite{TW} and showed that
\begin{align}\label{KT}
N_{K}(3,2,M) = 2\left(S_K(3,1)\right)^2q^{3M}M + O\left(q^{3M}\sqrt{M}\right)
\end{align}
where the implicit constant depends only on $K$. 

The final section is concerned with an application of Theorem \ref{MainResult0}. The analogy between integers and polynomials in one variable with coefficients in a finite field $\Fq$ can be extended to an analogy between positive integers and effective $0$-cycles on a variety $V$ over $\Fq$, i.e. formal sums $C = \sum_{i=1}^k n_i P_i$, where $n_i \in \mbN$ and $P_i$ are distinct closed points on $V$. In particular, primes correspond to closed points. In view of this idea, Chen \cite{Chen} establishes several results for $0$-cycles that can be seen as equivalent to classical theorems in analytic number theory. For example, the classical prime number theorem states that the probability that an integer between $e^m$ and $2e^m$ chosen uniformly at random to be prime is $\sim 1/m$ as $n \to \infty$. The analogue for $0$-cycles is given by \cite[Theorem 1]{Chen} and it says that for a geometrically connected variety $V$ over $\Fq$ of dimension $d \geq 1$, we have
\begin{equation}\label{eq:Chen}
\hspace{-0.4cm} \frac{\# \left\{ \text{closed points on $V$ of degree $m$} \right\}}{\# \left\{ \text{effective 0-cycles on $V$ of degree $m$} \right\}} \to \frac{1}{m\mathring{Z}(V,q^{-d})} + O\left(\frac{1}{mq^{m/2}}\right),
\end{equation}
as $m\to \infty$, where $\mathring{Z}(V, t) := Z(V, t)(1 - q^d t)$ and $Z(V, t)$ is the zeta function of $V$ over $\Fq$. 

Since an effective 0-cycle of degree $m$ on $V$ can be thought of as a point in $\Sym^mV(\Fq)$, we are inspired to consider a similar problem for $K$-points of $\Sym^m \mbP^2$, where $K$ is a finite extension of $\Fq(t)$, and thus, obtain a version of the prime number theorem for $0$-cycles on $\mbP^2$ over function fields. The prime 0-cycles of degree $m$ on $\mbP^2$ are precisely the $K$-points of degree $m$ of $\mbP^2$. As described in Section \ref{S:Geometry of the Hilbert scheme}, if $x \in \mbP^2(\overline{K})$ is of degree $m \geq 1$, then it has $m$ distinct Galois conjugates $x_1, \ldots, x_m$ and $\tilde{x} = \pi(x_1, \ldots, x_m)$ is a $K$-point of $\Sym^m \mbP^2$. Thus, primes correspond to the irreducible $K$-points of $\Sym^m \mbP^2$. We also notice that in this case, one needs to consider points of bounded height since the set of points $\Sym^m \mbP^2(K)$ is infinite. We obtain the following result. 

\begin{corollary}\label{C:PNT}
Let $m \geq 2$ and $K$ be a global function field of characteristic $>m$. Suppose that Manin\rq{}s conjecture holds for the irreducible points in $\Hilb^{m_0}\mbP^2(K)$ for all $m_0 \leq m$. Then, there exists a constant $c_m >0$ such that 
$$
\frac{\# \left\{ \text{prime effective $0$-cycles on $\mbP^2$ over $K$ of degree $m$} \right\}}{\# \left\{ \text{effective $0$-cycles on $\mbP^2$ over $K$ of degree $m$} \right\}} \to \frac{c_m}{M^{m-2}},
$$
$M\to \infty$.  
\end{corollary}
In general, $c_m$ is quite complicated, but an explicit expression is given in Section \ref{S:ZeroCycles}. When $m=2$, however, $c_2 = \frac{2}{3}$ and the result is actually unconditional.
\section{Geometry of the Hilbert scheme}\label{S:Geometry of the Hilbert scheme} 

Let $K$ be a global field. Given a smooth projective variety $V$ over $K$, we define the \emph{$m^{\text{th}}$-symmetric product} of $V$ to be the projective quotient variety 
$$
\Sym^m V = V^m/ \mfS_m,
$$
where the symmetric group $\mfS_m$ acts on $V^m$ by permutating the factors. Denote by 
$$
\pi : V^m \to \Sym^m V
$$
the canonical projection. Let $x \in V(\overline{K})$ be a point of degree $m \geq 1$, that is a point such that the degree of the residue field $\kappa(x)$ over $K$ is equal to $m$. The orbit of $x$ under the action of $\Gal(\overline{K}/K)$ contains exactly $m$ distinct points $x_1, \ldots, x_m$ which are the conjugates of $x$. Then, as in \cite{Rudulier}, let $\tilde{x} = \pi(x_1, \ldots, x_m) \in \Sym^mV(\overline{K})$. Since $\tilde{x}$ is invariant under $\Gal(\overline{K}/K)$, it is rational over $K$. Le Rudulier \cite[Definition 1.32]{Rudulier} defines points in $\Sym^mV(K)$ of the shape $\tilde{x}$ to be \emph{irreducible} and the rest to be \emph{reducible}. Furthermore, according to \cite[\S 6]{Beauville}, if 
$$
\Delta = \bigcup_{1 \leq i < j \leq m} \left\{(v_1, \ldots, v_m) \in V^m \mid v_i = v_j \right\}
$$ 
is the diagonal in $V^m$, then $\Sym^m V$ is singular along $D=\pi(\Delta)$. 

Suppose from now on that $V$ is a surface. In this case, by \cite[\S 6]{Beauville}, there is a resolution of singularities of $\Sym^m V$ given by the Hilbert-Chow birational morphism
$$
\varepsilon : \Hilb^m V \to \Sym^m V
$$
and $E = \varepsilon^{-1}(D)$ is an irreducible divisor of $\Hilb^m V $. Moreover, Fogarty \cite[Theorem 2.4]{Fogarty68} proves that $\Hilb^m V$ is a smooth, irreducible projective variety of dimension $2m$ and $\varepsilon$ induces an isomorphism 
\begin{align}\label{eq:isom}
\Hilb^m V - E \cong \Sym^m V - D.
\end{align}
The Picard group of $\Hilb^m V$ is also computed by Fogarty \cite[Theorem 6.2]{Fogarty73} (see \cite{Fogarty77} for characteristic 2) who proved that it has rank 2 and that 
\begin{equation}\label{eq:PicHilbm}
\Pic \left( \Hilb^m \mbP^2 \right) = \mbZ H \oplus \mbZ \frac{E}{2},
\end{equation}
where $H$ is the locus of schemes intersecting a fixed line and $E$ is exceptional divisor introduced above and corresponds to the locus of non-reduced schemes. This result allowed Huizenga \cite[Theorem 1.4]{Huizenga} to compute the cone of effective divisors for $\Hilb^m \mbP^2$. This is spanned by 
\begin{equation}\label{eq:mu}
\mu H -\frac{E}{2} \quad \text{and} \quad E,
\end{equation}
where $\mu$ is a certain minimum slope that will not be defined here. However, for $2\leq m \leq 171$, the values of $\mu$ can be found in \cite[Table 1]{Huizenga}, and by \cite[Remark 7.4]{Huizenga}, and if $m$ is of the form $\binom {r+2} 2$, then $\mu = r$. We remark that in this notation, the anticanonical divisor of $\Hilb^m \mbP^2$ is $\omega^{-1}_{\Hilb^m \mbP^2} = 3H$.

Moreover, due to results of Beauville \cite{Beauville} in characteristic 0 and Kumar--Thomsen \cite[Corollary 1]{KuTh} if $\ch(K)>m$, $\varepsilon$ is crepant. Thus, as noted in \cite[Proposition 3.3 and \S 3.2]{Rudulier}, the anticanonical line bundle on $\Hilb^m V$ given by
$
\omega^{-1}_{\Hilb^m V}=\varepsilon^*\omega^{-1}_{\Sym^m V}
$ 
is big since it is the pull-back of an ample divisor by a birational morphism. The anticanonical height on $\Sym^m V$ induces a height on $\Hilb^m V$ given by
$
H_{\omega^{-1}_{\Hilb^m V} } = H_{\omega^{-1}_{\Sym^m V} } \circ  \varepsilon.
$
Now, as in \cite[\S 1]{Rudulier}, given a point $v = \pi(v_1, \ldots, v_m)\in \Sym^m V(K)$, we define a height
$$
H_{\omega^{-1}_{\Sym^m V}}(v) = H_{\omega^{-1}_V}(v_1) \ldots H_{\omega^{-1}_V}(v_m).
$$

For the remainder of this section, let $K$ be a function field of degree $e$ over $\Fq(t)$ and $H_n$ be the usual absolute height on projective space $\mbP^{n}$ with respect to the ground field $\Fq(t)$ as in \eqref{eq:FctFieldHeight}. Regarding $\mbP^n$ over the ground field $K$, we have that for $x\in\mbP^n(\overline{K})$, the absolute height on projective space associated to the anticanonical line bundle is 
$$
H_{\omega^{-1}_{\mbP^n}}(x) = H_n(x)^{(n+1)e}.
$$

We are mainly interested in the case when $m=2$, $V = \mbP^2$ is defined over $K$ and $\ch(K) >2$. In order to study $K$-points on $\Hilb^2\mbP^2$ it is convenient to use the height function 
\begin{equation}\label{eq:H}
H(z) = H_{\omega^{-1}_{\Hilb^2 \mbP^2} }(z)^{1/3}.
\end{equation}

\begin{theorem}\label{MainResult1} 
Let $K$ be a function field of degree $e$ over $\Fq(t)$ of characteristic $>2$. Let $Z_0 = \varepsilon^{-1} \pi (\mbP^2 \times \mbP^2 (K))$. Suppose $M \geq 2(2g_K-1)$. Then, for all non-empty open subsets $U$ of $\Hilb^2 \mbP^2$ we have 
\begin{align*}
\# \left\{z \in \left( Z_0 \setminus E(K)  \right) \cap U(K) : H (z) = q^M \right\} = \frac{S_{K}(3,1)^2}{2}  q^{3M} M + O\left(q^{3M}\right),
\end{align*} 
where $S_K(3,1)$ is given by \eqref{eq:SchanuelsConstantFctFields}, and 
\begin{align*}
\# \left\{z \in  \Hilb^2 \mbP^2(K) \setminus Z_0  : H (z) = q^M \right\} = S_{K}(3,1)^2  q^{3M} M + O\left(\sqrt{M}q^{3M}\right),
\end{align*}
as $M\to \infty$. In particular, in the second part, the leading constant agrees with the prediction of Peyre and the implicit constant in the error term depends only on $K$.
\end{theorem} 

Thus, after removing the thin set $Z_0$, Manin's conjecture holds for $\Hilb^2\mbP^2$. The proof can be found in Section \ref{S:Proof}. The first part relies on \eqref{eq:TW} and \eqref{eq:TW2}, which is where the condition that $M\geq 2(2g_K-1)$ comes from, and the second part uses \eqref{KT}. Theorem \ref{MainResult0} follows from taking \eqref{eq:H} in Theorem \ref{MainResult1}

This can be seen as a more general function field analogue of Le Rudulier's result \cite[Theorem 4.2]{Rudulier} which states that there exists a thin set $Z_0 \subset \Hilb^2\mbP^2(\mbQ)$ such that for all non-empty open subsets $U$ of $\Hilb^2 \mbP^2$ we have 
\begin{align*}
\# \left\{z \in \left( Z_0 \setminus E(\mbQ)  \right) \cap U(\mbQ) : H_{ \omega^{-1}_{\Hilb^2 \mbP^2}, \mbQ} (z) \leq B \right\} \sim \frac{8}{\zeta(3)^2} B \log B,
\end{align*}
and 
\begin{align*}
\# \left\{z \in  \Hilb^2 \mbP^2(\mbQ) \setminus Z_0  : H_{ \omega^{-1}_{\Hilb^2 \mbP^2}, \mbQ} (z) \leq B \right\} \sim \frac{24+2\pi^2}{3\zeta(3)^2} B \log B,
\end{align*}
as $B \to \infty$.

\section{Manin's conjecture and the Peyre constant}\label{S:Peyre's constant}

Given a Fano variety $V$ defined over a function field $K$, let 
$$
N_V(M) := \# \left\{P \in V(K) : \xcal{H}(P) =q^M\right\},
$$ 
where $\xcal{H}$ denotes the anticanonical height $H_{\omega^{-1}_V}$. Define the anticanonical height zeta function of $V$ to be
$$
Z_{\xcal{H}}(s) = \sum_{y  \in V(K)} \frac{1}{\xcal{H}(y)^s}.
$$
Then, by an application of the Wiener--Ikehara theorem  \cite[Theorem 17.4]{Rosen} we obtain what is know as Manin\rq{}s conjecture, that is 
\begin{align}\label{eq:ManinsConjecture}
N_V(M) \sim c_{\xcal{H}}(V) \frac{(\log q)^{r_V}}{(r_V-1)!} q^M M^{r_V-1} , 
\end{align}
as $M\to \infty$, where
\begin{align}\label{eq:PeyresConjecture}
c_{\xcal{H} }(V) = \lim_{s\to 1} (s-1)^{r_V} Z_{\xcal{H}} (s).
\end{align}
and the rank $r_V$ of the Picard group of $V$ is equal to the the multiplicity of the pole of $Z_{\xcal{H}}(s)$ of $V$ at $s=1$. The leading constant above has been given a geometric interpretation by Peyre which can be found in \cite{Peyre95, Bourqui02, Bourqui, Peyre12}. Let $S$ a finite subset of the set of places $\Omega_K$ of $K$ containing all ramified places and infinite places. Then, the Peyre constant with respect to the anticanonical height $\xcal{H}=H_{\omega^{-1}_V}$ is given by
\begin{equation}\label{eq:PeyresConstant}
c_{\xcal{H}}(V) = \alpha^*(V)\beta(V) \tau_{\xcal{H}}(V).
\end{equation}
The first two factors are geometric invariants and are independent of the height. We have
\begin{equation}\label{eq:alphaPeyre}
\alpha^*(V) = \int_{C_\text{eff}^{\vee}(V)} e^{ - \langle \omega_V^{-1}, y \rangle} \diff y,
\end{equation}
where $C_\text{eff} $ 
is the effective cone of $V$, 
\begin{equation}\label{eq:CeffPeyre}
C_\text{eff}^{\vee}(V) = \left \{y \in \Pic(V)_{\mbR}^{\vee} \mid \langle x, y \rangle  \geq 0 \text{ for all } x \in C_{\text{eff}}(V) \right\},
\end{equation}
and $\diff y$ is the normalised Lebesgue measure on $\Pic(V)^{\vee}$. The second invariant is
\begin{equation}\label{eq:betaPeyre}
\beta(V) = \# H^1(K, \Pic(V^s)),
\end{equation}
where $V^s$ is the separable closure of $V$. The third factor is given by 
\begin{equation}\label{eq:tauPeyre}
\tau_{\xcal{H}}(V) = \omega_{\xcal{H}}(\overline{V(K)}) = C_{K,V} l_S(V)  \prod_{v \in S} \omega_{V,v} \prod_{v \not \in S} \lambda^{-1}_v \omega_{V,v},
\end{equation}
where $\lambda_v = L_v(1, \Pic(V^s))$ and
$$
C_{K,V} = \begin{cases}
\Delta_K^{-\dim(V)/2}, &\text{ if $K$ is a number field,}\\
q^{(1-g_K)\dim(V)}, &\text{ if $K$ is a function field},
\end{cases}
$$
and
$$
l_S(V) = \lim_{s\to 1} (s-1)^{\rk(\Pic(V^s))} \prod_{v \not\in S} L_v(s, \Pic(V^s)).
$$
From now on, let $K$ be a function field of degree $e$ over $\Fq(t)$. Given a place $v \in \Omega_K$, let $\kappa_v$ be the residue field with respect to $v$. If $v \not\in S$ and $\xcal{V}$ is a smooth projective model of $V$, then
$$
\omega_{V,v} = \frac{\#\xcal{V}(\kappa_v)}{(\#k_v)^{\dim(V)}}. 
$$

\begin{example}\label{Ex:Pn}
As in \cite[Exemple 1.31]{Rudulier}, in the case when $V= \mbP^n$ is defined over $K$, we have $\Pic(\mbP^n) \cong \mbZ$ and $\omega^{-1}_{\mbP^n} = \xcal{O}_{\mbP^n}(n+1)$. Thus, the first two factors are 
$$
\alpha^*(V) = \frac{1}{n+1}
\qquad
\text{and} 
\qquad
\beta(V) =1.
$$
Moreover, we can take $S = \emptyset$. We have $\#\kappa_v = q_v$ and
$$
\prod_{v \not\in S} L_v(s, \Pic(V^s)) = \prod_{v \in \Omega_K} \frac{1}{1- q_v^{-s}} = \zeta_K(s).
$$
Hence, we obtain that $l_S(\mbP^n)$ is equal to 
\begin{equation}\label{eq:LimZetaK(s)}
 	  \lim_{s\to 1}(s-1) L_K(q^{-s}) \zeta_{\Fq(t)}(s)
	= L_K(q^{-1})  \lim_{s\to 1}(s-1) \zeta_{\Fq(t)}(s) 
	=  \frac{J_K q^{1-g_K}}{(q-1) \log q},
\end{equation}
where $L_K$ is called the $L$-polynomial of $K/ \Fq$ and its value at $q^{-1}$ follows from \cite[Theorem 5.1.15]{Stichtenoth}, $g_K$ is the genus of $K$ and $J_K$ is the number of divisor classes of degree 0 (cardinality of the Jacobian) of $K$. Moreover, 
$$
\omega_{\mbP^n,v} = \frac{(q_v^{n+1} - 1)/(q_v-1)}{q_v^n} =  \frac{1-q_v^{1-n}}{1-q_v^{-1}},
$$
and the convergence factors are given by $\lambda_v^{-1} = L_v(1,\Pic(V^s))^{-1} = 1-q_v^{-1}$. 
Then 
$$\prod_{v \in \Omega_K} \lambda_v^{-1} \omega_v (V(K_v)) = \prod_{v \in \Omega_K} \left(1- q_v^{-1-n}\right)= \zeta_K(n+1)^{-1}.$$
Putting all this together, we obtain
$$
c_{\xcal{H}}\left(\mbP^n\right) =   \frac{S_K(n+1,1)}{\log q^{n+1}}, 
$$ 
where $S_K(n+1,1)$ is the function field version of Schanuel's constant in \eqref{eq:SchanuelsConstant}.
\end{example}

\begin{lemma}\label{Ex:HilbmP2}
Let $m \in \mbZ_{\geq 2}$ and $K$ be a degree $e \geq 1$ extension of $\Fq(t)$ such that $\ch(K)>m$. Then, Peyre\rq{}s constant for $\Hilb^m\mbP^2$ defined over $K$ is given by
$$
c_{\xcal{H}}(\Hilb^m\mbP^2)  = \frac{\mu J_K^2  }{9(q-1)^2 q^{2(m+1)(g_K-1)} (\log q )^2} \prod_{v \in \Omega_K} \left(1-q_v^{-1} \right)^2 \frac{\left | \Hilb^m\mbP^2(\mathbb{F}_{q_v}) \right |}{q^{2m}},
$$
where $g_K$ is the genus of $K$, $J_K$ is the number of divisor classes of degree $0$ which is the cardinality of the Jacobian of $K$, $\zeta_K$ is the zeta function of $K$ and $\mu$ is as in \eqref{eq:mu}.
\end{lemma}
\begin{proof}
By \eqref{eq:PicHilbm}, \eqref{eq:mu}, \eqref{eq:CeffPeyre}, and since the anticanonical divisor is $\omega^{-1}_{\Hilb^m\mbP^2}=3H$, we have that
$$
C_\text{eff}^{\vee}(\Hilb^m\mbP^2) = \left \{aH +\frac{b}{\mu}E, \left(a,\frac{b}{\mu}\right) \in \mbR^2 : a-\frac{b}{\mu}\geq 0,  \frac{b}{\mu} \geq 0 \right\},
$$
Then, by \eqref{eq:alphaPeyre}, and since $\mu \in \mbR_{>0}$, we obtain
$$
\alpha^*(\Hilb^m\mbP^2) = \int_{C_\text{eff}^{\vee}(\Hilb^m\mbP^2)} e^{ - \langle 3H, y \rangle} \diff y = \int_0^{\infty} \diff b \int_{\frac{b}{\mu}}^{\infty} e^{-3a} \diff a = \frac{\mu}{3^2}.
$$
By \eqref{eq:betaPeyre}, we have $\beta(\Hilb^m\mbP^2) = 1$. Since $\rk \Pic \Hilb^m \mbP^2 = 2$ and $\dim \Hilb^m \mbP^2 = 2m$, and by \eqref{eq:tauPeyre}, then 
$$
\tau_{\xcal{H}}(\Hilb^m\mbP^2) = q^{2m(1-g_K)} \lim_{s \to 1} (s-1)^2 \zeta_K(s)^2 \prod_{v \in \Omega_K} \left(1-q_v^{-1} \right)^2 \omega_v (\Hilb^m\mbP^2),
$$
where
$$
\omega_v (\Hilb^m\mbP^2) = \frac{ | \Hilb^m\mbP^2(\mathbb{F}_{q_v}) |}{{q_v}^{2m}}.
$$ 
As in \eqref{eq:LimZetaK(s)} in Example \ref{Ex:Pn}, we have that
$$
\lim_{s\to 1} (s-1)^2 \zeta_K(s)^2 = \left( \lim_{s\to 1} L_K(q^{-s}) \zeta_{\Fq(t)}(s) \right)^2 = \left(  \frac{J_K q^{1-g_K}}{(q-1) \log q} \right)^2.
$$
Thus, to compute $\tau_{\xcal{H}}(\Hilb^m\mbP^2)$ it is enough to know $| \Hilb^m\mbP^2(\mathbb{F}_{q})|$. A formula for this, involving Betti numbers, has been first given by Ellingsrud--Str{\o}mme \cite[Theorem 1.1]{ES}. However, it can be computed using a more general result of G\"ottsche \cite[Lemma 2.3.9]{G} which implies that 
$$
\sum_{m=0}^{\infty} \left |  \Hilb^m\mbP^2(\mathbb{F}_{q}) \right | t^m = \exp \left( \sum_{k=1}^{\infty} \frac{t^k}{k} \frac{ \left | \mbP^2 (\mathbb{F}_{q^k}) \right | }{1-q^k t^k}  \right).
$$
By \cite[Corollary 1.3]{ES}, we have that for $m \geq 2$, $\left |  \Hilb^m\mbP^2(\mathbb{F}_{q}) \right | = q^{2m} + 2q^{2m-1} + O\left(q^{2m-2} \right)$. It follows that $\omega_v (\Hilb^m\mbP^2)  = 1 + 2q_v^{-1} + O\left( q_v^{-2} \right)$, which is what we expect by Manin\rq{}s conjecture; hence, the product $\prod_v \lambda_v^{-1} \omega_v$, where $\lambda_v^{-1} = (1 - q_v^{-1})^2$, converges. We remark that, by the isomorphism \eqref{eq:isom}, we have
$$
\left | \Hilb^m\mbP^2(\Fq) \right | = \left | \Sym^m\mbP^2(\Fq) \right | -  \left | D(\Fq) \right | + \left | E(\Fq) \right |.
$$ 
Hence, we could also compute $\left |  \Hilb^m\mbP^2(\mathbb{F}_{q}) \right |$ by using Lemma \ref{L:Chen7}, if we understood the geometry of the exceptional divisor $E$, which has been studied independently by Iarrobino \cite{I} and Brian\c{c}on \cite{Briancon}. One can find an explicit description for the cases when $1 \leq m \leq 6$ in \cite[\S IV.2]{Briancon}. 
However, for $m> 2$, we could not find a nice interpretation for $\prod_{v \in \Omega_K} \left(1-q_v^{-1} \right)^2 \omega_v (\Hilb^m\mbP^2)$.
\end{proof}

\begin{corollary}\label{Ex:Hilb2P2}
Let $K$ be a degree $e \geq 1$ extension of $\Fq(t)$ such that $\ch(K)>2$. Then Peyre\rq{}s constant for $\Hilb^2\mbP^2$ defined over $K$ is given by
$$
c_{\xcal{H}}(\Hilb^2\mbP^2)  = \frac{S_K(3,1)^2}{9\left(\log q\right)^2},
$$
where $S_K(3,1)$ is given by \eqref{eq:SchanuelsConstantFctFields}.
\end{corollary}
\begin{proof} 
This follows by taking $m=2$ in Lemma \ref{Ex:HilbmP2}. By \cite[Table 1]{Huizenga}, we have that $\mu = 1$. Moreover, by \cite[Table 1]{ES}, we obtain that
$$
\omega_v (\Hilb^2\mbP^2) = \frac{\# \Hilb^2\mbP^2(\mathbb{F}_{q_v})}{q_v^4}=1 + 2q_v^{-1} + 3q_v^{-2} + 2q_v^{-3} + q_v^{-4},
$$
for all $v \in \Omega_K$. Thus, taking $\lambda_v^{-1} = (1 - q_v^{-1})^2 $, we obtain $\prod_v \lambda_v^{-1} \omega_v = \zeta_K(3)^{-2}$. Hence,
$$
c_{\xcal{H}}(\Hilb^2\mbP^2)  = \frac{J_K^2}{9(q-1)^2 q^{6(g_K-1)} (\log q )^2 \zeta_K(3)^2},
$$
and the result follows by \eqref{eq:SchanuelsConstantFctFields}.
\end{proof}

Similar to the case of number fields, Manin's conjecture and Peyre's prediction are compatible with products of varieties over function fields. This will be needed in the proofs of the results in Section \ref{S:ZeroCycles}.

\begin{theorem}\label{thm:products}
Let $V$, $W$ be two Fano varieties defined over a function field $K$ such that $V \times W(K) \neq \emptyset$ and $r_V \geq r_W$. Assume that \eqref{eq:ManinsConjecture} and \eqref{eq:PeyresConjecture} hold for $V$ and $W$. Then, we have that $N_{V\times W}(M) $ is
$$
\sim c_{H_{\omega_{V\times W}^{-1}}}(V\times W) \frac{(\log q)^{r_{V\times W}}}{(r_{V\times W}-1)!} q^M M^{r_{V\times W}-1}, 
$$
as $M\to \infty$, where
$$
c_{H_{ \omega_{V\times W}^{-1 } } }(V\times W) = \lim_{s\to 1} (s-1)^{r_{V\times W}}Z_{H_{\omega_{V\times W}^{-1}}} (s)
$$
agrees with the constant predicted by Peyre and the rank $r_{V\times W}$ of the Picard group of $V\times W$ is equal to the the multiplicity of the pole of the anticanonical height zeta function $Z_{H_{\omega_{V\times W}^{-1}}}(s)$ at $s=1$.
\end{theorem}

\begin{proof}
As in the proof of \cite[Lemma 3.0.2]{Peyre95}, we have an isomorphism
\begin{equation}\label{eq:PicProduct}
\Pic V \times \Pic W \to \Pic(V \times W).
\end{equation}
Thus, $r_{V \times W} = r_V + r_W$ and the metric on the anticanonical divisor $\omega^{-1}_{V \times W}$ is the product of the metrics on $\omega^{-1}_{V}$ and $\omega^{-1}_{W}$. Hence, the height of a point $(x,y) \in V\times W (K)$ is given by $H_{\omega_{V\times W}^{-1}}(x,y) =  H_{\omega_{V}^{-1}}(x)H_{\omega_{W}^{-1}}(y)$. This implies that we can write $N_{V\times W}(M)$ as
\begin{align}\label{eq:MT}
& \sum_{i=0}^M \# \left\{x \in V(K) : H_{\omega_{V}^{-1}}(x) =q^i\right\} \# \left\{y \in W(K) : H_{\omega_{W}^{-1}}(y) =q^{M-i}\right\} \\
	\sim & \frac{c_{H_{\omega_{V}^{-1}}}(V) c_{H_{\omega_{W}^{-1}}}(W)}{(r_V-1)!(r_W-1)!} (\log q)^{r_V+r_W} q^M \sum_{i=0}^M  i^{r_V-1} (M-i)^{r_W-1},\nonumber
\end{align} 
as $M\to \infty$. Replacing the sum over $i$ by an integral, we obtain that the above is 
$$ 
 \sim \frac{c_{H_{\omega_{V}^{-1}}}(V) c_{H_{\omega_{W}^{-1}}}(W)}{(r_V+r_W-1)!} (\log q)^{r_V+r_W} q^M M^{r_V+r_W-1}.
$$
Noting that $Z_{H_{\omega_{V\times W}^{-1}}} (s) = Z_{H_{\omega_{V}^{-1}}} (s)  Z_{H_{\omega_{W}^{-1}}} (s)$ together with the fact that $V$ and $W$ satisfy \eqref{eq:PeyresConjecture}, we get the expected main term in the asymptotic formula for $N_{V\times W}(M) $ as $M\to \infty$. 

It is left to show that the obtained constant agrees with the definition given in Section \ref{S:Peyre's constant}. By \cite[Lemma 3.0.2]{Peyre95}, we have 
$ \alpha^*(V )  \alpha^*(W )= \alpha^*(V \times W)$ and $\beta(V)\beta(W)=\beta(V\times W)$. 
This also holds in the case of function fields because $\alpha^*$ and $\beta$ are geometric invariants. Now let $S$ be a finite subset of the set of places $\Omega_K$ of $K$ containing all ramified places and the infinite place. By \eqref{eq:PicProduct}, we have that $L_v(1,\Pic V^s) L_v(1,\Pic W^s) = L_v(1,\Pic (V\times W)^s)$ for all $v \in \Omega_K$. We also note that since $V$ and $W$ are projective, we have that $\dim(V\times W) = \dim(V)+\dim(W)$. Thus, $\tau_{H_{\omega^{-1}_{V}}}(V)\tau_{H_{\omega^{-1}_{W}}}(W) $ is equal to
\begin{align*} 
& q^{(1-g_K)(\dim (V \times W))} \prod_{v \in S}\frac{\#\xcal{V}(\kappa_v)\#\xcal{W}(\kappa_v)}{(\#k_v)^{\dim(V\times W)}} \prod_{v \not\in S} \frac{\#\xcal{V}(\kappa_v)\#\xcal{W}(\kappa_v)}{L_v(1,\Pic (V\times W)^s)(\#k_v)^{\dim(V \times W)}},
\end{align*}
 which is exactly $\tau_{H_{\omega^{-1}_{V\times W}}}(V\times W)$. Hence,
$$
c_{H_{\omega_{V}^{-1}}}(V) c_{H_{\omega_{W}^{-1}}}(W) = \alpha^*(V\times W ) \tau_{H_{\omega^{-1}_{V\times W}}}(V\times W)  =  c_{H_{\omega_{V\times W}^{-1}}}(V \times W),
$$
as claimed.
\end{proof}

\section{Proof of the main theorem}\label{S:Proof}

In this section we prove Theorem \ref{MainResult1}. Recall that $[K:\Fq(t)]=e$, $\ch(K)>2$ and $Z_0=\varepsilon^{-1}\pi(\mbP^2 \times \mbP^2(K))$. 
For the proof of the first part, define the height function
$$
H_{{\mbP^2 \times \mbP^2}} : (x_1, x_2) \to H_2(x_1)^eH_2(x_2)^e.
$$
Then, as in \cite[Proposition 3.14]{Rudulier}, $\# \left\{z \in \left( Z_0 \setminus E(K)  \right) \cap U(K) : H (z) = q^M \right\}$ is equal to
\begin{align}
& \# \left\{z \in \varepsilon^{-1}\pi  \left( \mbP^2 \times \mbP^2 \right) (K) \cap (\Hilb^2\mbP^2-E)(K) \cap U(K) : H (z) = q^M \right\} \nonumber\\
= & \# \left\{z \in V_r \cap \varepsilon((\Hilb^2\mbP^2-E) \cap U) (K) : H (z) = q^M \right\}\nonumber \\
= & \frac{1}{2} \# \left\{(x,y) \in \left( \mbP^2 \times \mbP^2 \right)  (K) \cap \pi^{-1}\varepsilon((\Hilb^2\mbP^2-E) \cap U) (K) : H_{\mbP^2 \times \mbP^2} (x,y) = q^M \right\}, \nonumber
\end{align}
where $V_r = \pi \left( \mbP^2 \times \mbP^2 (K)\right)$ is the set of reducible points of $\Sym^2\mbP^2(K)$. Remark that $ \frac{1}{2} \# \left\{(x,y) \in \left(  \mbP^2 \times \mbP^2 \right) (K)  : H_{\mbP^2 \times \mbP^2} (x,y) = q^M \right\}$ can be written as
\begin{align}\label{eq:part1}
 & \frac{1}{2} \sum_{N=0}^M \# \left\{x \in \mbP^2 (K)  : H_2 (x) = q^{\frac{N}{e}} \right\}\# \left\{y \in \mbP^2 (K) : H_2 (y) = q^{\frac{M-N}{e}} \right\}. 
\end{align}
If $e=1$, then $K=\Fq(t)$ and 
$$
S_{\Fq(t)}(3,1) = \frac{(q^3-1)(1-q^{-2})}{q-1},
$$
by \eqref{eq:SchanuelsConstantFctFields}. Noting that
$$
\# \left\{ P \in \mbP^{n}(\Fq(t)) :  H_n(P) = q^M \right\}  = S_{\Fq(t)}(n+1,1) q^{(n+1)M},
$$
we have that \eqref{eq:part1} is equal to
\begin{align*}
& \frac{1}{2} \left( \frac{2(q^3-1)}{q-1}S_{\Fq(t)}(3,1)q^{3M} + \sum_{N=1}^{M-1} S_{\Fq(t)}(3,1)^2 q^{3M}  \right) \\
=& \frac{1}{2} \left( \frac{2}{1-q^{-2}}S_{\Fq(t)}(3,1)^2 q^{3M} + (M-1)S_{\Fq(t)}(3,1)^2 q^{3M}  \right) \\
=&\frac{S_{\Fq(t)}(3,1)^2}{2} q^{3M}M +  \frac{q^2+1}{2(q^2-1)} S_{\Fq(t)}(3,1)^2 q^{3M}.
\end{align*}
Otherwise, if $e\geq 1$, we split the sum over $N$ into three sums: the first sum runs over $0\leq N \leq 2g_K-2$, the second over $2g_K-1 \leq N \leq M-2g_K+1$, and the last over the remaining $N \leq M$. Thus, by the result of Thunder and Widmer \eqref{eq:TW} and \eqref{eq:TW2}, the first and third sum are each $\ll S_K(3,1)q^{3M}$, where the implicit constant depends on $e$ and $q$. For the second sum we note that both $N$ and $M-N$ are $\geq 2g_K-1$ and hence we obtain it is equal to
$$
\left( \frac{M+1}{2} -(2g_K-1)\right) S_K(3,1)^2 q^{3M} + O\left(S_K(3,1)q^{3M-3g_K+2}\right).
$$
Thus, in both cases above we obtain
$$ \frac{1}{2} \# \left\{(x,y) \in \left(  \mbP^2 \times \mbP^2 \right) (K)  : H (x,y) = q^M \right\} = \frac{S_{K}(3,1)^2}{2}  q^{3M} M + O\left(q^{3M}\right).$$

Moreover, the contribution to $ \frac{1}{2} \# \left\{(x,y) \in \left(  \mbP^2 \times \mbP^2 \right) (K)  : H_{\mbP^2 \times \mbP^2} (x,y) = q^M \right\}$ coming from a proper closed subset of $\mbP^2 \times \mbP^2$ is at most
\begin{align}\label{eq:ClosedSubsets}
&= \frac{1}{2} \sum_{N=0}^M \# \left\{x \in \mbP^1 (K)  : H_1(x) = q^{\frac{N}{e}}\right\}\# \left\{y \in \mbP^2 (K) : H_2 (y) = q^{\frac{M-N}{e}} \right\}.
\end{align}
Following the same argument as above, if $K=\Fq(t)$, \eqref{eq:ClosedSubsets} is equal to
$$
 \frac{q^2-1}{2(q-1)}S_{\Fq(t)}(3,1)q^{3M} + \frac{1}{2} S_{\Fq(t)}(2,1)S_{\Fq(t)}(3,1) q^{3M}  \sum_{N=1}^{M-1} q^{-N} +  \frac{q^3-1}{2(q-1)}S_{\Fq(t)}(2,1)q^{2M}.
$$
If $[K : \Fq(t)] >1$, we split the sum over $N$ into three sums as before which contribute $\ll S_K(3,1)q^{3M}$, $\sim S_K(3,1)S_K(2,1) q^{3M-2g_K+1}$, and $\ll S_K(2,1)q^{2M}$, respectively. Hence, the contribution from proper closed subsets of $\mbP^2 \times \mbP^2$ is $O(q^{3M})$ and this concludes the proof of the first part of the theorem.

The proof of the second part is similar to the argument over number fields. Using the isomorphism
$$
\Hilb^2 \mbP^2 \setminus E \xrightarrow{\sim} \Sym^2\mbP^2 \setminus D,
$$
we note that it suffices to study rational points of $\Sym^2\mbP^2 \setminus D$ of bounded height. Let $\overline{x} \in \mbP^2 (\overline{K})$ be the conjugate of the quadratic point $x$. Then we have that
$$
\# \left\{y \in \left( \Sym^2 \mbP^2 \setminus V_r \right)(K)  : H(y) = q^M \right\} 
$$
is equal to
$$
\frac{1}{2} \# \left\{(x, \overline{x}) \in \left( \mbP^2 \times \mbP^2\right) (\overline{K})   : [K(x): K] =2, H_{2}(x)^{e}H_{2}(\overline{x})^{e} = q^M \right\}.
$$
Since by \cite[Proposition 1.17]{Rudulier}, the height is invariant under Galois conjugation, this reduces to counting quadratic points in $\mbP^2$. Thus, it is equal to 
\begin{align*}
& \frac{1}{2} \# \left\{x \in \mbP^2  (\overline{\Fq(t)})  : [K(x): K] =2, H_{ 2}(x) = q^{\frac{M}{2e}} \right\} \\
=& S_{K}(3,1)^2  q^{3M} M + O\left(\sqrt{M}q^{3M}\right),
\end{align*}
by the result of Kettlestring and Thunder \eqref{KT}.

Now we verify that the leading constant we obtain in the second part agrees with the prediction of Peyre. Over function fields, Manin's conjecture predicts that for a Fano variety $X$ with associated anticanonical height $\xcal{H}$ we have that $\lim_{s\to 1} (s-1)^r \xcal{Z}_{X} (s)$, where $\xcal{Z}_X (s)$ is the anticanonical height zeta function associated to $X$ and $r= \rk \Pic X$, is equal to Peyre's constant $c_{\xcal{H}}(X)$ as given by \eqref{eq:PeyresConstant} and that the multiplicity of the pole of $\xcal{Z}_X$ at $s=1$ equals $r$. In our case $X(K) = {\Hilb^2\mbP^2(K)}^{\text{irr}}$, since we only consider the irreducible points, and $\rk \Pic \Hilb^2 \mbP^2=2$. The height zeta function corresponding to the height $H$ defined by \eqref{eq:H} satisfies 
\begin{equation}\label{eq:ZetaFcts}
Z^{\text{irr}}_X(3s) = \xcal{Z}^{\text{irr}}_X(s)
\end{equation}
and thus, $Z^{\text{irr}}_X$ has a double pole at $s=1$. Hence, we consider $Z^{\text{irr}}_X(s+2)$ which, by Wiener--Ikehara theorem \cite[Theorem 17.4]{Rosen}, leads us to expect
$$ 
q^{-2M} \# \left\{z \in X(K) : H(z)=q^M\right\} \sim (\log q)^2 q^M M \lim_{s\to 1}(s-1)^2 Z^{\text{irr}}_{X}(s+2),
$$
as $M\to \infty$. Moreover, Lemma \ref{Ex:Hilb2P2} together with \eqref{eq:ZetaFcts} implies that $c_H(X) = 9 c_{\xcal{H}}(X)$ and thus, the conjecture predicts that the number of $K$-points $z$ on $\Hilb^2 \mbP^2$ of height $H(z)=q^M$ is
$$ 
\sim \frac{S_{K}(3,1)^2}{ (\log q)^2} \cdot (\log q)^2 q^{3M}M,
$$
as $M\to \infty$, which concludes the proof.

\section{Analogy with $0$-cycles}\label{S:ZeroCycles}

In this section we will prove a refined version of Corollary \ref{C:PNT}. 

\begin{corollary}\label{C:PNT2}
Let $m \geq 2$ and $K$ be a global function field of characteristic $>m$. Suppose that Manin\rq{}s conjecture holds for the irreducible points in $\Hilb^{m_0}\mbP^2(K)$ for all $m_0 \leq m$. Then, there exists a constant $c_m >0$ such that 
$$
\frac{\# \left\{ \text{prime effective $0$-cycles on $\mbP^2$ over $K$ of degree $m$} \right\}}{\# \left\{ \text{effective $0$-cycles on $\mbP^2$ over $K$ of degree $m$} \right\}} \to \frac{c_m}{M^{m-2}},
$$
$M\to \infty$, 
where $c_2 = \frac{2}{3}$ and 
$$
c_m =
\frac{\mu 3^{m-2} \zeta_K(3)^2 m! (m-1)!}{S_K(3,1)^{m-2}} \prod_{v \in \Omega_K} \left(1-q_v^{-1} \right)^2 \frac{\left | \Hilb^m\mbP^2(\mathbb{F}_{q_v}) \right |}{q^{2m}},
$$
if $m \geq 3$.
\end{corollary}
When $m=2$, the result is actually unconditional due to Theorem \ref{MainResult1}.

\subsection{Improving a particular case over $\Fq$.} First, we will present an improvement to \cite[Theorem 1]{Chen} for the special case when $V=\mbP^2$. We begin with some technical lemmas.

\begin{lemma}\label{L:Chen7}
The number of effective $0$-cycles of degree $m$ on $\mbP^2$ over $\Fq$ is 
$$
\left(1+\frac{1}{q}\right) \sum_{i=0}^{k-1}(i+1)(q^{2(m-i)}+q^{2i+1}) +
\begin{cases}
\frac{m+2}{2}q^m, &\text{ if $m=2k$,}\\
 \frac{m+1}{2}q^{m-1}\left(q^2+q+1\right), &\text{ if $m=2k+1$.}
\end{cases}
$$
\end{lemma}

\begin{proof}
Effective $0$-cycles of degree $m$ on $\mbP^2$ over $\Fq$ correspond to $\Fq$-points on $\Sym^m\mbP^2$. By \cite[Remark 1.2.4]{G}, we have
\begin{align*}
\sum_{m=0}^{\infty} \left |  \Sym^m\mbP^2(\Fq) \right | t^m = \exp \left(  \sum_{k=1}^{\infty}  \left | \mbP^2(\mathbb{F}_{q^k}) \right | \frac{t^k}{k} \right) = Z(\mbP^2,t), 
\end{align*}
the zeta function of $\mbP^2$ over $\Fq$. This is $(1-q^2t)^{-1}(1-qt)^{-1}(1-t)^{-1}$. Thus, by using Taylor expansion at $t=0$ we obtain the claimed result.
\end{proof}

\begin{lemma}\label{L:Chen8}
The number of prime effective $0$-cycles of degree $m$ on $\mbP^2$ over $\Fq$ is 
$$
\begin{cases}
 \frac{1}{m} \left(q^{2m} - q^{\frac{m}{2}} \right), &\text{ if $m$ is even,}\\
 \frac{1}{m} \left(q^{2m} +q^m - q^{\frac{2m}{j}} - q^{\frac{m}{j}} \right), &\text{ if $m$ is odd,}
\end{cases}
$$
where $j$ is the smallest divisor of $m$ that is greater than 1.
\end{lemma}

\begin{proof}
By \cite[Definition 1.2.3, Remark 1.2.4(1)]{G}, the number of prime effective $0$-cycles of degree $m$ on $\mbP^2$ over $\Fq$ is 
\begin{equation}\label{eq:primes}
\frac{q^{2m}+q^m +1}{m} - \frac{1}{m}\sum_{\substack{r\mid m \\ r \neq m}} r \left | P_r(\mbP^2,\Fq) \right |,
\end{equation}
where $P_r(\mbP^2,\Fq)$ are the prime effective $0$-cycles of degree $r$ on $\mbP^2$ over $\Fq$. If $m$ is even, write $m=2k$. Then, $k$ is the largest proper divisor of $m$, and \eqref{eq:primes} becomes
$$
\frac{q^{2m}+q^m +1}{m} - \frac{1}{m} \left(q^{2k} + q^k + 1 - \sum_{\substack{s\mid k \\ s \neq k}} s \left | P_s(\mbP^2,\Fq) \right | \right) - \frac{1}{m} \sum_{\substack{r\mid m \\ r \neq k, m}} r \left | P_r(\mbP^2,\Fq) \right |,
$$
and hence the result. If $m$ is odd, the proof is similar.
\end{proof}

For the case when $V=\mbP^2$, \eqref{eq:Chen} states that the proportion of effective $0$-cycles of degree $m$ in $\mbP^2$ is $\frac{1}{m}\left(1-q^{-1}-q^{-2}+q^{-3}\right) +  O\left(\frac{1}{mq^{m/2}} \right)$. However, by carefully combining Lemmas \ref{L:Chen7} and \ref{L:Chen8}, we obtain the following improvement of the error term. 

\begin{theorem}\label{T:Chen1}
The proportion of prime effective $0$-cycles of degree $m$ on $\mbP^2$ out of all effective $0$-cycles of degree $m$ on $\mbP^2$ is
$$
\frac{1}{m}\left(1- q^{-1} - q^{-2} + q^{-3} \right) + O\left(\frac{1}{mq^m} \right),
$$
as $m\to\infty$.
\end{theorem}

\subsection{$0$-cycles on $\mbP^2$ over function fields.} In this section, we present an application of our main result which can be seen as an analogue of the prime number theorem for $0$-cycles on $\mbP^2$ over function fields. By Theorem \ref{MainResult0}, if $K$ is a function field of characteristic $>2$, then Manin's conjecture holds for the irreducible points of $\Hilb^2\mbP^2(K)$, which were introduced in Section \ref{S:Geometry of the Hilbert scheme}. Throughout this section, we will assume that Manin\rq{}s conjecture holds for the irreducible points in $\Hilb^m\mbP^2(K)$ for all $m\geq 3$, where $K$ is a function field of characteristic $>m$. Let 
$$
N_{{\Sym^m\mbP^2}^{\text{irr}}}(M) = \left \{x \in \Sym^m\mbP^2(K) : x \text{ irreducible}, H_{\omega^{-1}_{\Sym^m\mbP^2}}(x) = q^M   \right\}.
$$
By the isomorphism given in \eqref{eq:isom} and Theorem \ref{MainResult1}, we have
\begin{equation}\label{eq:irr}
N_{{\Sym^m\mbP^2}^{\text{irr}}}(M) \sim  \begin{cases}
\frac{1}{9}S_K(3,1)^2 q^M M, &\text{ if $m=2$,}\\
c_{\omega^{-1}_{\Hilb^m\mbP^2}}(\Hilb^m\mbP^2) \left(\log q\right)^2 q^M M, &\text{ if $m>2$,}
\end{cases}  
\end{equation}
as $M\to \infty$, since $\rk \Pic \Hilb^m \mbP^2 = \rk\Pic \mbP^2 +1 = 2$ for all $m\geq 2$. By understanding the various types of non-irreducible points of $\Sym^m\mbP^2$, we shall obtain the following main result, which together with \eqref{eq:irr} and Lemma \ref{Ex:HilbmP2}, implies Corollary \ref{C:PNT}. 

\begin{theorem}\label{T:SymmP2}
Let $m$ be an integer $\geq 2$ and $K$ be a global function field of characteristic $>m$. Suppose that Manin\rq{}s conjecture holds for the irreducible points in $\Hilb^{m_0}\mbP^2(K)$ for all $3 \leq m_0 \leq m$. Then,  
$$ 
N_{\Sym^m\mbP^2}(M) \sim \frac{S_K(3,1)^2}{6} q^M M, \text{ if $m=2$,}
$$
and
$$
N_{\Sym^m\mbP^2}(M) = \frac{S_K(3,1)^m}{ 3^m m!(m-1)!} q^M M^{m-1} + O(q^M M^{m-3}), \text{ if $m>2$.}
$$
as $M\to \infty$. 
\end{theorem}

To prove Theorem \ref{T:SymmP2}, we will require the following results.

\begin{lemma}\label{BigRed}
The number of reducible points in $\Sym^m \mbP^2(K)$, for $m\geq 2$, of the shape $\pi(v_1, \ldots, v_m)$, where $v_i \in \mbP^2(K)$ and $v_i \neq v_j$ for all $1 \leq i,j \leq m$, is  
$$
\sim \frac{S_K(3,1)^m}{ 3^m m!(m-1)!} q^M M^{m-1} ,
$$
 as $M\to \infty$.
\end{lemma}

\begin{proof}
Let $X = V^m$, where $V$ is a Fano variety over a function field $K$ satisfying Manin\rq{}s conjecture. By Theorem \ref{thm:products}, we have 
$$
N_{X}(M) \sim c_{H_{\omega_{X}^{-1}}}(X) \frac{(\log q)^{mr_{V}}}{(mr_V-1)!} q^M M^{mr_V-1},
$$
 as $M\to \infty$. Set $V=\mbP^2$. Since $\rk \Pic\mbP^2 = 1$, we have $\rk \Pic X = m$. Moreover, by Example \ref{Ex:Pn} and the proof of Theorem \ref{thm:products} we have
$$
c_{H_{ \omega_{X}^{-1 } }}\left(X\right) = c_{H_{ \omega_{\mbP^2}^{-1 } }}^m \left(\mbP^2\right)  = \frac{S_K(3,1)^m}{3^m \left(\log q\right)^m},
$$ 
which concludes the proof. 
\end{proof}

\begin{lemma}\label{L:Technical}
Let $m, t, j \in \mbZ_{>0}$, such that $t<m$, and $1\leq j <t$. We have 
$$
\sum_{i=0}^{M} q^{\frac{(t-j)i}{t}} i^{m-t} =\left( \frac{1}{\frac{t-j}{t}\log q} +\frac{1}{2} \right)q^{\frac{t-j}{t}M} M^{m-t} + O\left(q^{ \frac{t-j}{t} M} M^{m-t-1} \right)
$$
as $M\to \infty$.
\end{lemma}

\begin{proof}
Let $f(x) = q^{\frac{(t-j)x}{t}} x^{m-t}$. Then, replacing the sum on the left hand-side by an integral we obtain 
\begin{equation}\label{eq:EulerMaclaurin1}
\sum_{i=0}^{M} f(i) \sim \int_0^M f(i) \diff i + \frac{f(M)}{2} + \sum_{k=1}^{\infty} \frac{B_{2k}}{(2k)!} \left(f^{(2k-1)}(M) - f^{(2k-1)}(0) \right),
\end{equation}
where $B_{2k}$ is the $2k$-th Bernoulli number. We note
\begin{align*}
\int_0^M q^{\frac{(t-j)i}{t}} i^{m-t} \diff i= \left(\frac{-t}{(t-j)\log q} \right)^{m-t+1} \int_0^{- \log q^{\frac{(t-j)M}{t}}} v^{m-t}e^{-v} \diff v,
\end{align*}
via a change of variables. We can write this in terms of gamma functions as
$$
\left( \frac{- t}{(t-j)\log q} \right)^{m-t+1} \left( \Gamma (m-t+1) -  \Gamma \left(m-t+1, \frac{(t-j)\log q}{-t} M\right)  \right).
$$
Since $m-t \in \mbZ_{>0}$, the above is 
\begin{align*}
=& \left( \frac{- t}{(t-j)\log q} \right)^{m-t+1} (m-t)! \left( 1 -  q^{ \frac{t-j}{t} M} \sum_{k=0}^{m-t} \left(\frac{(t-j) \log q}{-t} \right)^{-k} \frac{M^k}{k!}  \right)\\
= &  \frac{t}{(t-j)\log q}     q^{ \frac{t-j}{t} M} M^{m-t}  + O \left(q^{ \frac{t-j}{t} M} M^{m-t-1} \right).
\end{align*}
Computing derivatives of $f$, we obtain that the sum over $k$ in \eqref{eq:EulerMaclaurin1} is $O \left(q^{ \frac{t-j}{t} M} M^{m-t-1} \right)$. Putting this together leads to the claimed result. 
\end{proof}

\begin{lemma}\label{L:Technical2}
For $k$ odd, we have
$$
\frac{1}{k!} \sum_{i=0}^{M} i(M-i)^k \sim \frac{1}{(k+2)!} M^{k+2},
$$
as $M\to \infty$.
\end{lemma}

\begin{proof}
By the binomial theorem, the left hand-side is 
\begin{align*}
\frac{1}{k!} \sum_{i=0}^{M} i \sum_{j=0}^k \binom k j (-i)^j M^{k-j} &= \frac{1}{k!}  \sum_{j=0}^k \binom k j (-1)^j  M^{k-j} \sum_{i=0}^{M} i^{j+1}\\
	& \sim \frac{M^{k+2} }{k!}  \sum_{j=0}^k \binom k j   \frac{(-1)^j}{j+2},
\end{align*}
which concludes the proof. 
\end{proof}

\begin{proof}[Proof of Theorem \ref{T:SymmP2}]
In the case when $m=2$, the number of irreducible points of $\Sym^2\mbP^2(K)$ is given by \eqref{eq:irr}. The  reducible points of $\Sym^2\mbP^2(K)$ are all of the type in Lemma \ref{BigRed}. Thus, the number of reducible points of $\Sym^2\mbP^2(K)$ is
$$ 
 \sim \frac{S_K(3,1)^2}{18} q^M M,
$$
as $M\to \infty$, which concludes the proof of the case $m=2$. 

Suppose from now that $m\geq 3$. The singular points of $\Sym^m \mbP^2(K)$ are points coming from the diagonal $\Delta \in  \left(\mbP^2\right)^m$, i.e.  
$$
D(K) = \bigcup_{1 \leq i < j \leq m} \left\{\pi(v_1, \ldots, v_m) \mid v_1, \ldots, v_m \in \mbP^2(K), v_i = v_j \right\},
$$ 
where $\pi$ is the projection $\left(\mbP^2\right)^m \to \Sym^m \mbP^2$ introduced in Section \ref{S:Geometry of the Hilbert scheme}. If $ m\geq 3$, by Theorem \ref{thm:products} , the contribution to $N_{ \Sym^m \mbP^2 } (M)$ coming from these points involves counting points $v \in \left(\mbP^2\right)^{m-2}(K)$ and $v_{m-1} \in \mbP^2(K)$ such that $H_{\omega^{-1}_{\left(\mbP^2\right)^{m-2}}}(v) \left(H_{\omega^{-1}_{\mbP^2}}(v_{m-1})\right)^2=q^M$. Thus we have 
\begin{align*}
\sim& \left(c_{\omega^{-1}_{\mbP^2}}(\mbP^2)\right)^{m-1}  q^{\frac{M}{2}} \frac{\left(\log q\right)^{m-1}}{2 m!(m-3)!} \sum_{i=0}^M q^{\frac{i}{2}} i^{m-3} (M-i)\\
=& \frac{S_K(3,1)^{m-1} }{2 m! (m-3)!3^{m-1} } q^{\frac{M}{2}} \left( M \sum_{i=0}^M q^{\frac{i}{2}} i^{m-3} - \sum_{i=0}^M q^{\frac{i}{2}} i^{m-2} \right),
\end{align*}
by Example \ref{Ex:Pn} and Lemma \ref{Ex:Hilb2P2}. Now, by Lemma \ref{L:Technical}, the terms of order $M^{m-2}$ cancel out and we obtain that the above is 
$$
\sim  \frac{2S_K(3,1)^{m-1} }{3^{m-1} m! (m-3)! \left(\log q\right)^2} q^{M} M^{m-3}.
$$

The non-singular points of $\Sym^m \mbP^2(K)$ are points $\pi(v_1, \ldots, v_m)$, where $v_1, \ldots, v_m \in \mbP^2(\overline{K})$ are all distinct and can be partitioned into $c_1$ Galois conjugacy cycles of length 1, $c_2$ cycles of length $2$, $\ldots$, $c_m$ cycles of length $m$, such that $\sum_{i=1}^m ic_i = m$. We remark that a point $v_i$ in a conjugacy cycle of length $k$ is a point in $\mbP^2(\overline{K})$ such that $[K(v_i):K]=k$ and the other points in the cycle are its distinct Galois conjugates. It is convenient to analyse all these cases depending on how many cycles of length $>1$ there are.

\emph{0 cycles of length $>1$.} This implies $c_1 = m$ and , so these are precisely the reducible points of the shape $\pi(v_1, \ldots, v_m)$, where $v_i \in \mbP^2(K)$ and $v_i \neq v_j$ for all $1 \leq i,j \leq m$, whose contribution is given by Corrolary \ref{BigRed} and is
$$
\sim \frac{S_K(3,1)^m}{ 3^m m!(m-1)!} q^M M^{m-1} ,
$$
 as $M\to \infty$.
 
\emph{1 cycle of length $>1$.} Thus, $c_1=m-j$ and $c_j=1$, where $2\leq j \leq m$. The case when $j=m$  encompasses exactly the irreducible points and their contribution is given by \eqref{eq:irr}. Now fix j such that $2\leq j < m$. These are points $\pi(v_1, \ldots, v_m)$ such that $v_1, \ldots, v_{m-j} \in \mbP^2(K)$, $v_{m_j+1} \in \mbP^2(\overline{K})$, $[ K(v_{m_j+1}):K ] = j$, and $v_{m_j+1}, \ldots, v_m$ are the $j$ distinct Galois conjugates of $v_{m_j+1}$. Counting such points with height $H_{\omega^{-1}_{\Sym^m\mbP^2}}(\pi(v_1, \ldots, v_m)) = q^M $ corresponds to counting $v \in \left(\mbP^2\right)^{m-j}(K)$ with $H_{\omega^{-1}_{\left(\mbP^2\right)^{m-j}}}(v)=q^i$ and $w \in \mbP^2(\overline{K})$ such that $[K(w):K]=j$ and $H_{\omega^{-1}_{\mbP^2}}(w)^j=q^{M-i}$, for $0\leq i \leq M$. Thus, it is equal to
\begin{equation}\label{eq:Contrib(m-j)(j)}
\frac{j! }{m!} \sum_{i=0}^M N_{\left(\mbP^2\right)^{m-j}}(i)   N_{{\Sym^j\mbP^2}^\text{irr}}\left(\frac{M-i}{j}\right) .
\end{equation}
By Manin\rq{}s conjecture for the irreducible points in $\Hilb^m\mbP^2$ and Theorem \ref{thm:products}, we expect that, as $L\to \infty$, there are
$$
\sim \left(c_{\omega^{-1}_{\mbP^2}}\left(\mbP^2\right)\right)^{m-j} c_{\omega^{-1}_{\Hilb^j\mbP^2}}\left(\Hilb^j\mbP^2\right) \frac{\left(\log q\right)^{m-j+2}}{(m-j+1)!} q^L L^{m-j+1} 
$$
points $(v,w) \in \left(\mbP^2\right)^{m-j} \times \Sym^j\mbP^2 (K)^{\text{irr}}$ of height $H_{\omega^{-1}_{\left(\mbP^2\right)^{m-j}}}(v) H_{\omega^{-1}_{\Sym^j\mbP^2}}(w) = q^L$, since the rank of the Picard group of the product variety is $m-j+2$. If $2< j <m$, this is at most $\sim L^{m-2}$, and only the case $j=2$ gives $\sim L^{m-1}$. This implies that the number of points given by \eqref{eq:Contrib(m-j)(j)} is at most $O\left(M^{m-2}\right)$ for $2< j <m$, and thus, does not contribute to the main term in $N_{\Sym^m\mbP^2}(M)$. In the case when $j=2$, we obtain that that \eqref{eq:Contrib(m-j)(j)} is 
\begin{align*}
\sim& \frac{2 }{m!}   \left(c_{\omega^{-1}_{\mbP^2}}(\mbP^2)\right)^{m-2} c_{\omega^{-1}_{\Hilb^2\mbP^2}}\left(\Hilb^2\mbP^2\right) \frac{\left(\log q\right)^m}{(m-3)!} \sum_{i=0}^M q^{i + \frac{M-i}{2}} i^{m-3} \frac{M-i}{2}\\
=& \frac{S_K(3,1)^m }{3^m m! (m-3)!} q^{\frac{M}{2}}\left( M \sum_{i=0}^M q^{\frac{i}{2}} i^{m-3} - \sum_{i=0}^M q^{\frac{i}{2}} i^{m-2} \right),
\end{align*}
by Example \ref{Ex:Pn} and Lemma \ref{Ex:Hilb2P2}. Now, by Lemma \ref{L:Technical}, the terms of order $M^{m-2}$ cancel out and we obtain that the above is 
$$
\sim  \frac{4S_K(3,1)^m }{3^m m! (m-3)! \left(\log q\right)^2} q^{M} M^{m-3}.
$$
Through a similar method we can show that the contribution from the cases when $2<j<m$ is in fact at most $O\left( M^{m-4} \right)$.

\emph{k cycles of length $>1$, where $1 < k \leq \lfloor m/2 \rfloor$.} This is a generalisation of the previous case. Thus, we have $c_1=m-2k$ and $\sum_{i=2}^m c_i = k$. We expect that the number of points $x \in \left(\mbP^2\right)^{m-2k}(K) \times \left(\Sym^{2}\mbP^2\right)^{j_1} (K)^{\text{irr}} \ldots \times \left( \Sym^{m}\mbP^2 \right)^{j_m} (K)^{\text{irr}}$, where $\sum_{i=2}^m i j_i = 2k$, of height $H(x)=q^L$, where $H$ is the product of the anticanonical heights, is 
$$
\left(c_{\omega^{-1}_{\mbP^2}}(\mbP^2)\right)^{m-2k} \prod_{i=2}^m \left( c_{\omega^{-1}_{\Hilb^{i}\mbP^2}}\left(\Hilb^{i}\mbP^2\right) \right)^{j_i} \frac{\left(\log q\right)^{m-2k+2l}}{(m-2k+2l-1)!} q^L L^{m-2k+2l-1}, 
$$
as $L\to \infty$, since the rank of the Picard group of the product variety is $m-2k+2l$, where $l= j_2 + \ldots + j_m$. A simple calculation shows that only in the case when $j_2=k$ and $j_3= \ldots = j_{m} = 0$ the above is $\sim L^{m-1}$ and in all other cases we have at most $L^{m-2}$ points. Thus we analyse the former case. The number of such points with height $H_{\omega^{-1}_{\Sym^m\mbP^2}}(\pi(v_1, \ldots, v_m)) = q^M $ is  
\begin{align}\label{eq:Contrib(m-2k)(k)}
\frac{2^k}{m!} & \sum_{i=0}^M N_{\left(\mbP^2\right)^{m-2k}}(i) \sum_{i_1=0}^{M-i} N_{{\Sym^2\mbP^2}^\text{irr}}\left(\frac{i_1}{2}\right)  \ldots \sum_{i_{k-1}=0}^{M-i-i_1 - \ldots i_{k-2}}     \nonumber \\
	& \times N_{{\Sym^2\mbP^2}^\text{irr}}\left(\frac{i_{k-1}}{2}\right) N_{{\Sym^2\mbP^2}^\text{irr}}\left(\frac{M-i-i_1-\ldots -i_{k-1}}{2}\right).
\end{align}
Denote $M-i-i_1-\ldots - i_{k-2}$ by $L_1$. Then, the last sum in \eqref{eq:Contrib(m-2k)(k)} is
\begin{align*}
 \sim & \left( c_{\omega^{-1}_{\Hilb^2\mbP^2}}\left(\Hilb^2\mbP^2\right) \right)^2 \frac{\left(\log q\right)^4}{2^2} q^{\frac{L_1}{2}} \sum_{i=0}^{L_1}  i_{k-1}(L_1-i_{k-1})\\
    \sim & \left( c_{\omega^{-1}_{\Hilb^2\mbP^2}}\left(\Hilb^2\mbP^2\right) \right)^2 \frac{\left(\log q\right)^4}{2^2} q^{\frac{L_1}{2}} \frac{{L_1}^3}{3!},
\end{align*}
by Lemma \ref{L:Technical2}. We iterate this procedure for the sums in \eqref{eq:Contrib(m-2k)(k)} starting with the sum over $i_{k-2}$ up to the sum over $i_{i_1}$ to obtain that \eqref{eq:Contrib(m-2k)(k)} is 
\begin{align*}
\sim & \frac{2^k}{m!} \sum_{i=0}^M N_{\left(\mbP^2\right)^{m-2k}}(i) \left( c_{\omega^{-1}_{\Hilb^2\mbP^2}}\left(\Hilb^2\mbP^2\right) \right)^k \frac{\left(\log q\right)^{2k}}{2^k} q^{\frac{M-i}{2}} \frac{(M-i)^{2k-1}}{(2k-1)!}\\
\sim & \frac{ \left(c_{\omega^{-1}_{\mbP^2}}(\mbP^2)\right)^{m-2k} \left( c_{\omega^{-1}_{\Hilb^2\mbP^2}}\left(\Hilb^2\mbP^2\right) \right)^k \left(\log q\right)^{m} q^{\frac{M}{2}}}{m!(m-2k-1)!(2k-1)!} \sum_{i=0}^M q^{\frac{i}{2}} i^{m-2k-1}(M-i)^{2k-1}.
\end{align*}
By the binomial theorem, the sum over $i$ above is equal to
\begin{align*}
=& \sum_{j=0}^{2k-1} \binom {2k-1} j (-1)^{j}  M^{2k-1-j} \sum_{i=0}^M q^{\frac{i}{2}} i^{m-2k-1+j}\\
=&  \left(\frac{2}{\log q} +\frac{1}{2} \right)q^{\frac{M}{2}} M^{m-2} \sum_{j=0}^{2k-1} \binom {2k-1} j (-1)^{j}  + O\left( q^{\frac{M}{2}}M^{m-3} \right),
\end{align*}
by Lemma \ref{L:Technical}. Noting that $\sum_{j=0}^{2k-1} \binom {2k-1} j (-1)^{j}  =0$, we get a contribution $O\left( M^{m-3} \right)$. Through a similar method we can show that the contribution coming from the other choices of $c_i$'s is in fact at most $O\left( M^{m-4} \right)$.

In conclusion, the main contribution to $N_{\Sym^m\mbP^2}(M)$ comes from reducible points of the type described in Lemma \ref{BigRed}. 
\end{proof}

\end{document}